\documentclass[11pt,reqno,tbtags,a4paper]{amsart}
\usepackage{amssymb}
\usepackage{url}
\usepackage[square,numbers]{natbib}
\bibpunct[, ]{[}{]}{;}{n}{,}{,}

\title[Random replacements in P{\'o}lya urns]
{Random replacements in P{\'o}lya urns with infinitely many colours}

\date{27 November, 2017}

\author{Svante Janson}
\thanks{Partly supported by the Knut and Alice Wallenberg Foundation}
\address{Department of Mathematics, Uppsala University, PO Box 480,
SE-751~06 Uppsala, Sweden}
\email{svante.janson@math.uu.se}
\urladdr{http://www.math.uu.se/svante-janson}

\subjclass[2010]{} 

\numberwithin{equation}{section}

\renewcommand\le{\leqslant}
\renewcommand\ge{\geqslant}

\allowdisplaybreaks

\theoremstyle{plain}
\newtheorem{theorem}{Theorem}[section]
\newtheorem{lemma}[theorem]{Lemma}

\theoremstyle{definition}
\newtheorem{example}[theorem]{Example}

\newtheorem{remark}[theorem]{Remark}

\theoremstyle{remark}

\newenvironment{romenumerate}[1][-10pt]{
\addtolength{\leftmargini}{#1}\begin{enumerate}
 \renewcommand{\labelenumi}{\textup{(\roman{enumi})}}%
 }{\end{enumerate}}

\newcounter{oldenumi}
\newenvironment{romenumerateq}
{\setcounter{oldenumi}{\value{enumi}}
\begin{romenumerate} \setcounter{enumi}{\value{oldenumi}}}
{\end{romenumerate}}

\newcounter{thmenumerate}

\newcounter{xenumerate}   

\newcommand{\refT}[1]{Theorem~\ref{#1}}

\newcommand{\refL}[1]{Lemma~\ref{#1}}

\newcommand{\refS}[1]{Section~\ref{#1}}
\newcommand{\refSS}[1]{Section~\ref{#1}}

\newcommand{\refE}[1]{Example~\ref{#1}}

\begingroup
  \divide\count255 by 60
  \multiply\count255 by -60
  \advance\count255 by \time
  \ifnum \count255 < 10 \xdef\klockan{\the\count1.0\the\count255}
  \else\xdef\klockan{\the\count1.\the\count255}\fi
\endgroup

\newcommand\set[1]{\ensuremath{\{#1\}}}
\newcommand\bigset[1]{\ensuremath{\bigl\{#1\bigr\}}}

\newcommand\bigpar[1]{\bigl(#1\bigr)}

\newcommand\Bigabs[1]{\Bigl|#1\Bigr|}

\def\rompar(#1){\textup(#1\textup)}    

\def\xexp(#1){e^{#1}}

\newcommand\ntoo{\ensuremath{{n\to\infty}}}

\newcommand\punkt{.\spacefactor=1000}    
\newcommand\iid{i.i.d\punkt}    
\newcommand\ie{i.e\punkt}
\newcommand\eg{e.g\punkt}

\newcommand{\as}{a.s\punkt}
\newcommand{\aex}{a.e\punkt}

\newcommand\eqd{\overset{\mathrm{d}}{=}}

\newcommand\bbR{\mathbb R}

\newcommand\bbZ{\mathbb Z}

\newcommand\bbZgeo{\mathbb Z_{\ge0}}
\newcommand\bbRgeo{\mathbb R_{\ge0}}

\newcounter{CC}
\newcounter{cc}

\renewcommand\P{\operatorname{\mathbb P{}}}

\newcommand\PD{\operatorname{PD}}

\newcommand\gd{\delta}

\newcommand\gf{\varphi}

\newcommand\gl{\lambda}

\newcommand\gs{\sigma}

\newcommand\cA{\mathcal A}
\newcommand\cB{\mathcal B}
\newcommand\cC{\mathcal C}

\newcommand\cF{\mathcal F}

\newcommand\cL{{\mathcal L}}
\newcommand\cM{\mathcal M}
\newcommand\cN{\mathcal N}

\newcommand\cP{\mathcal P}

\newcommand\cR{{\mathcal R}}
\newcommand\cS{{\mathcal S}}
\newcommand\cT{{\mathcal T}}

\newcommand\tR{{\tilde R}}
\newcommand\tS{{\tilde S}}

\newcommand\tX{{\widetilde X}}

\newcommand\etta{\boldsymbol1}

\newcommand\qw{^{-1}}

\newcommand\oi{\ensuremath{[0,1]}}

\newcommand\ooo{[0,\infty)}

\newcommand\oooo{[0,\infty]}
\newcommand\setoi{\set{0,1}}

\newcommand\dd{\,\mathrm{d}}

\newcommand\lhs{left-hand side}
\newcommand\rhs{right-hand side}

\newcommand\citetyear[1]{\citet{#1} (\citeyear{#1})}

\newcommand\cMx{\cM_*}
\newcommand\cNx{\cN_*}
\newcommand\cMpm{\cM_\pm}
\newcommand\gsf{$\gs$-field}
\newcommand\Uoi{\mathrm{U}(0,1)}
\newcommand\sh{^\sharp}
\newcommand\phix{\phi\sh}
\newcommand\psix{\psi\sh}
\newcommand\RR{\widehat R}
\newcommand\xoo{_0^\infty}
\newcommand\tensor{\otimes}
\newcommand\pk{probability kernel}

\newcommand\mdot{\cdot}
\newcommand\cmo{\cM_0}
\newcommand\tcmo{\widetilde\cM_0}

\newcommand{\Polya}{P\'olya}

\begin{document}

\begin{abstract} 
We consider the general version of P{\'o}lya urns 
recently studied by 
Bandyopadhyay and Thacker (2016+) 
and
Mailler and  Marckert (2017), 
with the space of colours being any Borel space $S$ and the state of the urn
being a finite measure on $S$.
We consider urns with random replacements, and show that these can be
regarded as urns with deterministic replacements using the colour space
$S\times\oi$.
\end{abstract}

\maketitle

\section{Introduction}\label{S:intro}

The  original \Polya{} urn,
studied already in \citeyear{Markov1917} by \citet{Markov1917}
but later named after 
\Polya{} who studied it in 
\citetyear{EggPol1923} and \citetyear{Polya1930},
 contains balls of two colours.
At discrete time steps, a ball is drawn
at random from the urn (uniformly), and it is replaced together with $a$
balls of the 
same colour, where $a\ge1$ is some given constant.
The (contents of the) 
urn is thus a Markov process $(X_n)\xoo$, with state space $\bbZgeo^2$.
(The initial state $X_0$ is some arbitrary given non-zero state.)

This urn model has been generalized by various authors in a number of ways,
all keeping the basic idea of a Markov process of sets of balls of
different colours (types), where balls are drawn at random and the drawn
balls determine the next step in the process.
(The extensions are all
usually called \Polya{} urns, or perhaps generalized \Polya{} urns.)
These generalizations have been studied by a large number of authors, and
have found a large number of applications, see for example 
\cite{Mahmoud}, \cite{SJ154}, \cite{BTnew}, \cite{MM}
and the references given there.
The extensions include (but are not limited to) the following, 
in  arbitrary combinations.

\begin{romenumerate}

\item 
The number of different colours can be any finite integer $d\ge2$.
The state space is thus $\bbZgeo^d$.
\item \label{UR}
The new balls added to the urn can be of any colours.
We have a (fixed) replacement matrix $(R_{i,j})_{i,j=1}^d$ of non-negative
integers; when a ball of colour $i$ is drawn, it is replaced together with
$R_{i,j}$ new balls of colour $j$, for every $j=1,\dots,d$.
\item \label{Urand}
The replacements can be random. Instead of a fixed replacement matrix as in
\ref{UR}, we have for each colour $i$ a random vector $(R_{i,j})_{j=1}^d$.
Each time a ball of colour $i$ is drawn, 
replacements are made according to 
a new copy of this vector, independent of everything that has happened so far.
\item \label{Ureal}
The ``numbers of balls'' of different colours can be arbitrary non-negative
real numbers (which can be interpreted as the amount or mass of each colour).
The state space is thus $\bbRgeo^d$, and the replacement matrix $(R_{i,j})$
in \ref{UR}, or its random version in \ref{Urand}, has arbitrary entries in
$\bbRgeo$. 
\item \label{U<0}
Balls may also be removed from the urn. This means that $R_{i,j}$ in
\ref{UR} or \ref{Urand} may be negative. (Some conditions are required in
order to guarantee that we never remove balls that do not exist; the state
space is still $\bbZgeo^d$ or $\bbRgeo^d$.)
The simplest case,
which frequently appears in applications, 
is drawing without replacement;
then $R_{i,i}=-1$ is allowed but 
$R_{i,j}\ge0$ when $i\neq j$, 
this means that the drawn ball is
not replaced (but balls of other colours are added).
\end{romenumerate}

In contrast to the many papers on \Polya{} urns with a finite number of
colours,
there has so far been very few studies of extensions to infinitely many
colours. 
One example is \citet{BTinfty,BTrate} who studied the case when the space
of colours is $\bbZ^d$, and the replacements are translation invariant.
A very general version of \Polya{} urns
was introduced by \citet{BlackwellMcQ} 
in a special case (with every replacement having the colour of the drawn
ball, as in the original \Polya{} urn, see \refE{E1}),
and 
much more generally 
(with rather arbitrary deterministic replacements) 
by \citet{BTnew} and \citet{MM};
this version can be described by:
\begin{romenumerateq}
\item\label{Upolish} 
The space $S$ of colours is a measurable space.
The state space is now the space $\cM(S)$ of finite measures on $S$;
if the current state is $\mu$, then the next ball is drawn with the
distribution
$\mu/\mu(S)$.
\end{romenumerateq}

This version seems very powerful, and can be expected to find many
applications in the future.

\begin{remark}
  Note that the case when $S$ is finite in \ref{Upolish} is equivalent to
  the version \ref{Ureal}.
Also with an infinite $S$ in \ref{Upolish}, a state $\mu\in\cM(S)$ of the
process can be interpreted as the amount of different colours in the urn.
(The amount is thus now described by a measure; note that 
the measure may be diffuse,
meaning that each single colour has mass 0). 
\end{remark}
\begin{remark}\label{Rborel}
The colour space $S$ is assumed to be a Polish topological space 
in \cite{BlackwellMcQ} and \cite{MM}, and for the convergence results
in \cite{BTnew}, while the representation results in \cite{BTnew} are stated
for a general $S$.
We too make our definitions for an arbitrary measurable space $S$, but
we restrict to Borel spaces in our main result.
(This includes the case of a Polish space, 
see \refL{Lborel} below. Our results do not use any topology on $S$.)
\end{remark}

The purpose of the present note is to show that this model 
with a measure-valued \Polya{} urn
and the results for it
by \cite{BTnew} and \cite{MM}
extend almost automatically to the case of random replacements,
at least in the case with no removals.
In fact, we show that the model is so flexible that a random replacement 
can be seen as a deterministic replacement using the larger colour space
$S\times\oi$, where the extra coordinate is used to simulate the randomization.
Random replacement in this general setting was raised as an open problem in
\cite{MM}, and our results 
together with the results of \cite{MM}
thus answer this question.

We give a precise definition of the
measure-valued
version of  \Polya{} urns with random
replacement in \refS{Sdef}.
We include there 
a detailed treatment of measurability questions,  showing that
there are no such problems.
(This was omitted in 
\cite{BTnew} and \cite{MM}, where 
the situation is simpler and straightforward.
In our, technically more complex, situation,
there is a need to  verify  measurability explicitly.)

The main theorem is the following. The proof is given in \refS{SpfT1}.

\begin{theorem}\label{T1}
Consider a measure-valued \Polya{} urn process $(X_n)\xoo$
in  a Borel space $S$,
with random replacements.
Then there exists a \Polya{} urn process $(\tX_n)\xoo$ in 
$S\times \oi$
with deterministic replacements 
such that $\tX_n=X_n\times \gl$ and thus $X_n=\pi\sh(\tX_n)$ for every $n\ge0$,
where 
$\gl$ is the Lebesgue measure,
$\pi:X\times\oi\to X$ is the projection,
and $\pi\sh$ the corresponding mapping of measures.
\end{theorem}

Urns without replacement or with other removals,
 see \ref{U<0},
are treated in \refS{Sneg}.
We show that \refT{T1} holds in this case too,
but the result in this case is less satisfactory than in the case without
removals, and it cannot be directly applied to extend the results for this
case in \cite{MM}, see \refS{Sneg}.

\begin{remark}
Many papers, including \cite{BTnew} and \cite{MM},
consider only \emph{balanced} \Polya{} urns,
\ie, urns where the total number of balls added to the urn each time
is deterministic, and thus the total number of balls in the urn after $n$
steps
is a deterministic linear function of $n$;
in the measure-valued context, this means that the total mass $X_n(S)=an+b$,
where $b=X_0(S)$. (We may without loss of generality assume $a=1$ by rescaling.)
We have no need for this assumption in the present paper.
\end{remark}

\section{Preliminaries}\label{Sprel}

We state some more or less well-known definitions and facts, adding a few
technical details. 

\subsection{Measurable spaces}\label{SSspaces}

A measurable space $(S,\cS)$ is a set $S$ equipped with a \gsf{} $\cS$ of
subsets of $S$. We often abbreviate $(S,\cS)$ to $S$ when the \gsf{} is evident.
When $S=\oi$ or another Polish topological space (i.e., a
complete metric space), we tacitly assume $\cS=\cB(S)$, the Borel \gsf{}
generated by the open subsets.

For a measurable space $S=(S,\cS)$, let $\cM(S)$ be the set of finite measures
on $S$, let $\cMx(S):=\set{\mu\in\cM(S):\mu\neq0}$ and 
$\cP(S):=\set{\mu\in\cM(S):\mu(S)=1}$, the set of probability measures
on $S$;
furthermore, let $\cMpm(S)$ be the space of finite signed measures on $S$.
We regard $\cMpm(S)$, $\cM(S)$, $\cMx(S)$ and $\cP(S)$
as measurable spaces, equipped with the \gsf{} generated by
the mappings $\mu\mapsto \mu(B)$, $B\in\cS$;
note that $\cMx(S)$ and $\cP(S)$ are measurable subsets of $\cM(S)$.
(See \eg{} \cite[Chapter 1, p.~19]{Kallenberg}, but note that $\cM(S)$ there
is larger than ours.)

If $f\ge0$ is a measurable function on a measurable space $S$
and $\mu$ is a measure on $S$,
let $\mu(f):=\int_S f\dd\mu\in[0,\infty]$.
Note that the mapping $\mu\mapsto\mu(f)$ is measurable $\cM(S)\to[0,\infty]$
for every fixed $f\ge0$.

If $S$ and $T$ are measurable spaces, and $\gf:S\to T$ is a measurable mapping,
then, as in \refT{T1} above,
 $\gf\sh:\cM(S)\to\cM(T)$ denotes the induced mapping of measures, defined
by $\gf\sh(\mu)(B)=\mu(\gf\qw(B))$ for $\mu\in\cM(S)$ and $B\in\cT$.

If $X$ is a random element of $S$, its distribution is an element
of $\cP(S)$, denoted by  $\cL(X)$.

A signed measure $\mu\in\cMpm(S)$ has a Jordan decomposition
$\mu=\mu^+-\mu^-$ with 
$\mu^+,\mu^-\in\cM(S)$, and the variation of $\mu$ is
$|\mu|=\mu^++\mu^-\in\cM(S)$; see \cite[Chapter 4]{Cohn}.

\subsection{Borel spaces}
A \emph{Borel space} is a measurable space that is isomorphic to a Borel
subset of $\oi$.
This can be reformulated by the following standard result.
\begin{lemma}\label{Lborel}
  The following are equivalent for a measurable space $(S,\cS)$, and thus
each property  characterizes Borel  spaces.
  \begin{romenumerate}
  \item 
$(S,\cS)$ is isomorphic to a Borel subset of a Polish space.
  \item 
$(S,\cS)$ is isomorphic to a Borel subset of $\oi$.
  \item 
$(S,\cS)$ is isomorphic to a Polish space.
\item 
$(S,\cS)$ is isomorphic to a compact metric space.
  \item 
$(S,\cS)$ is either countable (with all subsets	measurable), 
or isomorphic to $\oi$.
  \end{romenumerate}
\end{lemma}
For a proof, see \eg{} \cite[Theorem 8.3.6]{Cohn} or 
\cite[Theorem  I.2.12]{Parthasarathy}.
An essentially equivalent statement is that any two Borel  spaces
with the same cardinality are isomorphic.

In \refT{T1}, we consider only Borel spaces; 
\refL{Lborel} shows that this is no great loss of generality for
applications.

\begin{lemma}\label{Labs}
  If $S$ is a Borel space, then the mappings 
$\mu\mapsto|\mu|$,
$\mu\mapsto\mu^+$
and
$\mu\mapsto\mu^-$  
are measurable $\cMpm(S)\to\cM(S)$.
In particular, $\cM(S)$ is a measurable subset of $\cMpm(S)$.
\end{lemma}

\begin{proof}
By \refL{Lborel}, we may assume that $S$ is a Borel subset of $\oi$. 
Then, for every $B\in\cS$ and $\mu\in\cMpm(S)$,
\begin{equation}
  |\mu|(B) = 
\lim_{n\to\infty}
\sum_{i=0}^n\Bigabs{\mu\bigpar{B\cap[\tfrac{i}{n},\tfrac{i+1}{n})}}, 
\end{equation}
which shows that $\mu\mapsto|\mu|(B)$ is measurable.
Hence $\mu\mapsto|\mu|$ is measurable.
Furthermore, $\mu^\pm=\frac12(|\mu|\pm\mu)$, and
$\cM(S)=\set{\mu\in\cMpm(S):\mu^-=0}$.
\end{proof}

\begin{lemma}\label{LMB}
  If $S$ is a Borel space, then $\cMpm(S)$, 
$\cM(S)$, $\cMx(S)$ and $\cP(S)$ are Borel  spaces. 
\end{lemma}
\begin{proof}
$\cM(S)$ is Borel as a special case of \cite[Theorem 1.5]{Kallenberg-rm}.
Alternatively,
by \refL{Lborel}, we may assume that $S$ is a compact metric space with its
Borel \gsf.
Then, see \eg{} \cite[Theorem A2.3]{Kallenberg},
$\cM(S)$ is a Polish space, and its Borel \gsf{} equals the \gsf{} defined
above for $\cM(S)$; hence, $\cM(S)$ is a Borel space.

Next, $\cMx(S)$ and $\cP(S)$ are measurable subsets of $\cM(S)$ and thus
also Borel spaces.

Finally, the Jordan decomposition $\mu\mapsto(\mu^+,\mu^-)$ gives a
bijection
\begin{equation}
\psi:  \cMpm(S)\leftrightarrow
\cM':=
\bigset{(\mu_1,\mu_2)\in\cM(S)^2:|\mu_1-\mu_2|(S)=|\mu_1|(S)+|\mu_2|(S)}.
\end{equation}
\refL{Labs} shows that $\psi$ is measurable, and so is trivially
its inverse $\psi\qw:(\mu_1,\mu_2)\mapsto\mu_1-\mu_2$.
Moreover, it follows from \refL{Labs} that the set $\cM'$ is a measurable
subset of $\cM(S)^2$; hence $\cM'$ is a Borel space, and thus so is $\cMpm(S)$.
\end{proof}

\begin{remark}
  \refL{Labs} may fail if $S$ is not a Borel space.
For a  counter-example, let $S=\set{0,1}^\bbR$,
define for $A\subseteq\bbR$ the \gsf{} $\cS_A$ on $S$
consisting of all sets $\pi\qw(B)$ where $\pi:S\to S_A:=\setoi^A$ is the
projection and $B$ is a Borel set in $\setoi^A$, and let
$\cS:=\bigcup\bigset{\cS_A:A\text{ countable}}$.
(This is the Baire \gsf{} on $S$.)
Then any measurable function $F(\mu)$
on $\cMpm(S)$ is a function of $(\mu(B_i))\xoo$ for some sequence $B_i\in\cS$,
which means that $B_i\in\cS_{A_i}$ for some countable $A_i$.
Choose some $x\in\bbR\setminus\bigcup_i A_i$, and define
$s_0,s_1\in S=\setoi^\bbR$ by $s_j(x)=j$ and $s_j(y)=0$ when $y\neq x$.
Let $\nu:=\gd_{s_0}-\gd_{s_1}$. Then $\nu(B_i)=0$ for every $i$, and thus
$F(\nu)=F(0)$. Hence, $F(\mu)$ cannot equal $|\mu|(S)$ for every
$\mu\in\cMpm(S)$. Consequently, $\mu\mapsto|\mu|(S)$ is not measurable on
$\cMpm(S)$.
Similarly, $\cM(S)$ is not a measurable subset of $\cMpm(S)$.
\end{remark}

\subsection{Kernels}\label{SSkernels}
(See \eg{} 
\cite[pp.~20--21, 106--107, 116 and 141--142]{Kallenberg}.)
Given two measurable spaces $S=(S,\cS)$ and $T=(T,\cT)$,
a \emph{kernel} from $S$ to $T$ is a 
measurable mapping $\mu:S\to\cM(T)$.
We write the mapping as $s\mapsto\mu_s$; thus a kernel is,
equivalently, a family $\set{\mu_s}_{s\in S}$ of finite measures on $T$
such that $s\mapsto \mu_s(B)$ is measurable for every $B\in\cT$.
It follows that if $\mu:S\to\cM(T)$ is a kernel and
$f:T\to\ooo$ is measurable, then $s\mapsto\mu_s(f)$ is measurable
$S\to[0,\infty]$.
A \emph{probability kernel} is a kernel that maps $S$ into $\cP(T)$,
\ie, a kernel $\mu$
such that $\mu_s$ is a probability measure for every $s\in S$.

If $\mu$ is a probability kernel from $S$ to $T$ 
and $\nu$ is a probability measure on $S$, then a probability
measure $\nu\tensor\mu$ is
defined on $S\times T$ by
\begin{equation}\label{tensor}
  \nu\tensor\mu(B)=\int_S\dd\nu(s)\int_T\etta_{B}(s,t)\dd\mu_s(t),
\qquad B\in \cS\times \cT.
\end{equation}
Note that if the random element $(X,Y)\in S\times T$  
has the distribution $\nu\tensor\mu$, then the marginal distribution of $X$
is $\nu\in\cP(S)$;  
we denote the marginal distribution of $Y$ by $\nu\mdot\mu\in\cP(T)$.

If $X$ and $Y$ are random elements of $S$ and $T$, respectively, then a 
\emph{regular conditional distribution} of $Y$ given $X$
is a probability kernel $\mu$ from $S$ to $T$  such that
for each $B\in \cT$,
$\P\bigpar{Y\in B\mid X} = \mu_X(B)$ a.s.
(I.e., $\mu_X(B)$ is a version of the conditional expectation 
$\P\bigpar{Y\in B\mid X}$.)
This is easily seen to be equivalent to:
$\mu$ is a probability kernel such that 
$(X,Y)$ has the distribution
$\cL(X)\tensor\mu$ given by \eqref{tensor}.

If $\mu$ is a \pk{} from a measurable space $S$ to itself, and
$\mu_0\in\cP(S)$ is any distribution, we can iterate \eqref{tensor} and
define, for any $N\ge1$, 
a probability measure $\mu_0\tensor\mu\tensor\dotsm\tensor\mu$ on
$S^{N+1}$ such that if $(X_0,\dots,X_N)$ has this distribution, then
$X_0,\dots,X_N$ is a Markov chain with initial distribution $X_0\sim\mu_0$
and transitions given by the kernel $\mu$, \ie,
$\P\bigpar{X_n\in B\mid X_0,\dots,X_{n-1}}
= \P\bigpar{X_n\in B\mid X_{n-1}}=\mu_{X_{n-1}}(B)$
for any $B\in\cS$ and $1\le n\le N$.
Moreover, 
these finite Markov chains extend to  
an infinite Markov chain $X_0,X_1,\dots$ with the transition kernel $\mu$.

\begin{remark}
The existence of an  infinite Markov chain follows 
without any condition on $S$
by a theorem by
Ionescu Tulcea
\cite[Theorem 6.17]{Kallenberg}.
(If $S$ is a Borel space, we may also, as an alternative,
use Kolmogorov's theorem
\cite[Theorem 6.16]{Kallenberg}.)

The construction of an infinite Markov chain extends to any sequence of
different measurable spaces $S_0,S_1,\dots$ 
and \pk{s} $\mu_i$ from $S_{i-1}$ to $S_i$, $i\ge1$, but we need here only the
homogeneous case.
\end{remark}

\subsection{Two lemmas}

\begin{lemma}\label{L1}
Let $S=(S,\cS)$ and $T=(T,\cT)$ be measurable spaces.
A map $\mu:S\to\cM(\cM(T))$ is a kernel 
from $S$ to $\cM(T)$
if and only if, for every
bounded
measurable function $h:T\to\ooo$, 
\begin{equation}\label{mus}
  s\mapsto \int_{\cM(T)} e^{- \nu(h)}  \dd\mu_s(\nu)
\end{equation}
is measurable on $S$.
\end{lemma}

\begin{proof}
If $h:T\to\ooo$ is measurable, then $\nu\to e^{-\nu(h)}$ is measurable
$\cM(T)\to\oi$. 
Hence, if $\mu$ is a kernel, then \eqref{mus} is measurable.

Conversely, 
let $B_+(T)$ be the set of all bounded measurable $h:T\to\ooo$ and
assume that \eqref{mus} is measurable for every $h\in B_+(T)$.
Let $\cA$ be the set of bounded measurable functions $F:\cM(T)\to\bbR$ such
that
$s\mapsto \int_{\cM(T)} F(\nu)\dd\mu_s(\nu)$ is measurable $S\to\bbR$.
Furthermore, 
if $h\in B_+(T)$, let
$\Psi_h:\cM(T)\to\oi$ be the function $\Psi_h(\nu)=e^{-\nu(h)}$,
and let $\cC:=\bigset{\Psi_h:h\in B_+(T)}$.
The assumption says that
$\cC\subset\cA$.
Furthermore, $\cC$ is closed under multiplication, since 
$\Psi_{h_1}\Psi_{h_2}=\Psi_{h_1+h_2}$.
It follows by the  monotone class theorem,
in \eg{} the version given in \cite[Theorem A.1]{SJIII},
that $\cA$ contains every bounded function that is measurable with respect
to the \gsf{} $\cF(\cC)$ generated by $\cC$.

Let again $h\in B_+(T)$.
Then, for every $\nu\in\cM(T)$,
\begin{equation}
  n\bigpar{\Psi_0(\nu)-\Psi_{h/n}(\nu)}
= n\bigpar{1-e^{-\nu(h)/n}}
\to \nu(h)
\qquad\text{ as \ntoo}.
\end{equation}
Hence the mapping $\nu\mapsto h(\nu)$ is $\cF(\cC)$-measurable.
In particular, taking $h=\etta_B$, it follows that $\nu\mapsto \nu(B)$ is 
$\cF(\cC)$-measurable for every $B\in \cT$.
Since these maps generate the \gsf{} of $\cM(T)$, 
it follows that if $D\subseteq\cM(T)$ is measurable, then 
$\etta_D$ is $\cF(\cC)$-measurable, and thus $\etta_D\in\cA$.
This means that
\begin{equation}
  s\mapsto\int_{\cM(T)}\etta_D(\nu)\dd\mu_s(\nu)=\mu_s(D)
\end{equation}
is measurable for all such $D$, which means that $s\mapsto\mu_s$ is measurable.
\end{proof}

We shall also use the following lemma from \cite{Kallenberg}.

\begin{lemma}[{\cite[Lemma 3.22]{Kallenberg}}]\label{L2}
Let $(\mu_s)_{s\in S}$ be a probability kernel from a measurable space $S$ to   
a Borel space $T$. Then there exists a measurable function $f:S\times\oi\to
T$ such that if\/ $U\sim\Uoi$, then 
$f(s,U)$ has the distribution $\mu_s$ for every $s\in S$.
\end{lemma}

\section{\Polya{} urns}\label{Sdef}

In this section, we give formal definitions of the \Polya{} urn model with
an arbitrary colour space $S$.
The state
space of the urn process is $\cM(S)$, or more precisely $\cMx(S)$, since the
process gets stuck and stops when there is no ball left in the urn.

In this section we consider for simplicity only urns 
with replacement and no removals, \ie, all replacements are positive.
See \refS{Sneg} for the more general case.

We treat first the deterministic case defined and studied by
 \cite{BTnew} and \cite{MM}; our model is the same as theirs and we add only
 some technical details as a preparation for the random replacement case.

\subsection{Deterministic replacements}\label{SSdefD}

The replacements are described by a \emph{replacement kernel},
which is a kernel $R=(R_s)_{s\in S}$ from $S$ to itself, 
\ie, a measurable map $S\to \cM(S)$;
the interpretation is that if we draw a ball of colour $s$, then it is
returned together with an additional measure $R_s$.
More formally, we define, for $\mu\in\cMx(S)$, 
a function $\phi_\mu:S\to\cMx(S)$ by
\begin{equation}\label{phimu}
  \phi_\mu(s):=\mu+R_s;
\end{equation}
thus if the composition of the urn is described by
the measure $\mu$, and we draw a ball of
colour $s$, then the new composition of the urn is
$\phi_\mu(s)$.
Moreover, 
the ball is drawn with distribution
$\mu':=\mu/\mu(S)$. Hence, letting $\phix_\mu:\cP(S)\to\cP(\cMx(S))$ denote the
mapping of probability measures induced by $\phi_\mu$, 
the composition after the draw has the distribution
\begin{equation}\label{RR}
\RR_\mu:= \phix_\mu\bigpar{\mu'}
= \phix_\mu\bigpar{\mu/\mu(S)}\in\cP(\cMx(S)).
\end{equation}

\begin{lemma}\label{LD}
  The mapping $\mu\mapsto\RR_\mu$
defined by  \eqref{phimu}--\eqref{RR}
is  a measurable map $\cMx(S)\to\cP(\cMx(S))$,
\ie,
a probability kernel from $\cMx(S)$ to itself.
\end{lemma}

\begin{proof}
We use \refL{L1}.
Let $h:S\to\ooo$ be measurable.
Then, by \eqref{RR},
\begin{align}\label{RR+}
  \int_{\cM(S)}e^{-\nu(h)}\dd \RR_\mu(\nu)
&=
  \int_{S}e^{-\phi_\mu(s)(h)}\dd \mu'(s)
=
 \int_{S}e^{-(\mu(h)+R_s(h))}\dd \mu'(s)
\notag
\\&
=
\frac{e^{-\mu(h)}}{\mu(S)} \int_{S}e^{-R_s(h)}\dd \mu(s) .   
\end{align}
  Since $R$ is a kernel, $s\mapsto R_s(h)$ is measurable, 
and thus $\mu\mapsto \int_{S}e^{-R_s(h)}\dd \mu(s)$
is a measurable function on $\cM(S)$.
Hence, \eqref{RR+} shows that 
the \lhs{}
$  \int_{\cM(S)}e^{-\nu(h)}\dd \RR_\mu(\nu)$
is a measurable function of $\mu\in\cMx(S)$, and thus \refL{L1} shows that
$\RR:\cMx(S)\to\cM(\cM(S))$
is a kernel from $\cMx(S)$ to $\cM(S)$.

Since  $\mu\in\cMx(S)$ implies $\phi_\mu(s)\in\cMx(S)$  by \eqref{phimu}, 
$\RR$ is also a kernel from $\cMx(S)$ to itself.
Finally, $\RR$ is a probability kernel, since 
$\RR_\mu$ is a probability measure by \eqref{RR}.
\end{proof}

The \Polya{} urn process $(X_n)\xoo$ is the Markov process
with values in $\cMx(S)$ defined as in \refSS{SSkernels}
by the probability kernel $\RR$ and an arbitrary
initial state $X_0\in\cMx(S)$.
(In general $X_0$ may be random, but we assume for simplicity that  $X_0$ is
deterministic; this is also the case in most applications.)

\begin{example}\label{E1}
We illustrate the definition with a classical example.

  Let $S$ be any measurable space and let the replacement kernel be
  $R_s=\gd_s$, i.e., $R_s(B)=\etta_B(s)$ for $s\in S$ and $B\in\cS$.
This means that the drawn ball is returned together with another ball of the
same colour.
(Note that $\gd_s$ is well defined even if $\set{s}\notin \cS$.)

With $S=\set{0,1}$ and $X_0$ an integer-valued measure, 
this is the urn studied by 
\citet{Markov1917},
\citet{EggPol1923} and \citet{Polya1930}.

The case when
$S$ is an arbitrary Polish space
and $X_0\in\cM(S)$ is arbitrary was
studied by \citet{BlackwellMcQ};
they showed that
$X_n/X_n(S)$
\as{} converges (in total variation) to a random discrete probability measure,
with a so called Ferguson distribution.
See also \citet[Exercises 2.2.6 and 0.3.2, and Section 3.2]{Pitman} 
(the case $S=\oi$, which is no loss of generality by \refL{Lborel}),
which imply that the limit can be represented as $\sum_i P_i\gd_{\xi_i}$
with $\xi_i$ \iid{} with distribution $X_0/X_0(S)$ and $(P_i)_1^\infty$ with the
Poisson--Dirichlet distribution $\PD(0,X_0(S))$.
By \refL{Lborel}, the result of \cite{BlackwellMcQ} 
extends to any Borel space $S$. 
In fact, the result holds for an arbitrary measurable space $S$; this can 
for example be seen by considering the same process on $S\times\oi$, 
starting with $X_0\times\gl$, regarding the second coordinate as labels
and using the result for $\oi$; we omit the details.
\end{example}

\subsection{Random replacement}\label{SSdefR}

For the more general version with random replacement, 
the replacement measures $R_s$, $s\in S$ are random.
We let 
$\cR_s:=\cL(R_s)\in\cP(\cM(S))$ for every $s\in S$; $\cR_s$ is thus the
distribution of the replacement, and we assume that
$s\mapsto\cR_s$ is a given probability kernel $S\to\cP(\cM(S))$.
This means that $\phi_\mu(s)$ in \eqref{phimu} is a random measure in
$\cMx(S)$,
with a distribution that we denote by $\Phi_\mu(s)\in\cP(\cMx(S))$.
Note that for a fixed $\mu\in\cMx(S)$,
the map $\psi_\mu:\nu\mapsto\mu+\nu$ is measurable $\cM(S)\to\cMx(S)$, 
and thus induces a 
measurable map $\psix_\mu:\cP(\cM(S))\to \cP(\cMx(S))$; furthermore,
\begin{equation}
  \label{kok}
\Phi_\mu(s)=\psix_\mu(\cR_s).
\end{equation}
Hence, $s\mapsto\Phi_\mu(s)$ is a probability kernel from $S$ to $\cMx(S)$.

If we draw from an urn with composition $\mu\in\cMx(S)$, then the drawn
colour $s$ has as above
distribution $\mu':=\mu/\mu(S)$,
and the resulting urn has thus a distribution $\RR_\mu$ 
that is the corresponding
mixture of the distributions $\Phi_\mu(s)$,
\ie, 
in  the notation of \refSS{SSkernels},  
see \eqref{tensor} and the comments after it,
\begin{equation}\label{kqk}
  \RR_\mu=\mu'\mdot\Phi_\mu.
\end{equation}

\begin{lemma}\label{LR}
  The mapping $\mu\mapsto\RR_\mu$
defined by  \eqref{kok}--\eqref{kqk}
is  a measurable map $\cMx(S)\to\cP(\cMx(S))$,
\ie,
a probability kernel from $\cMx(S)$ to itself.
\end{lemma}

\begin{proof}
Let $h:S\to\ooo$ be measurable.
Then, extending \eqref{RR+} in the deterministic case,
by \eqref{kqk} and  \eqref{kok},
\begin{equation}\label{RRR}
  \begin{split}
  \int_{\cM(S)}e^{-\nu(h)}\dd \RR_\mu(\nu)
&=
\int_S  \int_{\cM(S)}e^{-\nu(h)}\dd \Phi_\mu(s)(\nu) \dd\mu'(s)
\\&
=
\int_S  \int_{\cM(S)}e^{-\psi_\mu(\nu)(h)}\dd \cR_s(\nu) \dd\mu'(s)
\\&
=
\frac{e^{-\mu(h)}}{\mu(S)}\int_S  \int_{\cM(S)}e^{-\nu(h)}\dd \cR_s(\nu) \dd\mu(s).
  \end{split}
\end{equation}
Here $\nu\mapsto e^{-\nu(h)}$ is a measurable function $\cM(S)\to\ooo$,
and thus
$s\mapsto \int_{\cM(S)}e^{-\nu(h)}\dd \cR_s(\nu)$ is a measurable function
$S\to\oooo$. 
Consequently, the \rhs{} of \eqref{RRR} is a measurable function of $\mu$,
and \refL{L1} shows that $\RR$ is a  kernel. 
The proof is completed as the proof of \refL{L2}.
\end{proof}

The \Polya{} urn process $(X_n)\xoo$ is,
as in the deterministic case above,
the Markov process
with values in $\cMx(S)$ defined
by the probability kernel $\RR$.

\section{Proof of \refT{T1}}\label{SpfT1}

Let $U\sim\Uoi$.
By \refL{L2}, there exists a measurable function $f:S\times\oi\to\cM(S)$
such that $f(s,U)\sim \cR_s$ for every $s$; i.e., $f(s,U)\eqd R_s$.
In other words, we can use $f(s,U)$ as the replacement measure $R_s$
for the urn $(X_n)\xoo$.

Let $\tS:=S\times\oi$ and define 
\begin{equation}
  \tR_{s,u}:=f(s,u)\times \gl \in \cM(\tS). 
\end{equation}
The mapping $\mu\mapsto\mu\times\gl$ is measurable $\cM(S)\to\cM(\tS)$;
hence $\tR_{s,u}$ is measurable, and thus a kernel from $\tS$ to itself.

We now let $(\tX_n)\xoo$ be the \Polya{} urn process in $\cM(\tS)$ defined
by the replacement kernel $\tR$, with initial value $\tX_0=X_0\times \gl$.
We claim that we can couple the processes such that 
$\tX_n=X_n\times\gl$ for every $n\ge0$.
We prove this by induction.
Given $\tX_n=\mu\times\gl$, we draw a ball $(s,u)$ with  the
distribution $(\mu\times\gl)'=\mu'\times\gl$, which means that $s$ has
distribution $\mu'$ and $u$ is uniform and independent of $s$;
hence, given $s$, $\tR_{s,u}=f(s,u)\times\gl$ has the same distribution as
$f(s,U)\times\gl\eqd R_s\times \gl$. 
We may thus assume (formally by the transfer theorem 
\cite[Theorem 6.10]{Kallenberg})
that $\tR_{s,u}=R_s\times\gl$, and thus 
$\tX_{n+1}=\tX_n+\tR_{s,u}=(X_n+R_s)\times\gl=X_{n+1}\times\gl$.
\qed

\section{Urns without replacement or with other subtractions}
\label{Sneg} 

The models in \refS{Sdef} can easily be extended to 
urns without replacement or with removals (subtractions) of other balls.

\subsection{Deterministic replacements}\label{SSdefD-}
In the deterministic case, we 
let the replacements $R_s$ be given by a measurable map $S\to\cMpm(S)$.
We assume that 
we are given some measurable subset $\cmo$ of $\cMx(S)$
such that 
for every $\mu\in\cmo$,
\begin{equation}\label{rso}
  R_s+\mu \in\cmo
\quad\text{for $\mu$-\aex{} $s$}.
\end{equation}
I.e., by \eqref{phimu}, $\phi_\mu(s)\in\cmo$ 
$\mu$-\aex{},
which means that $\RR_\mu$ in \eqref{RR} is a probability measure on 
$\cmo\subseteq\cMx(S)$.
 \refL{LD} is modified to say that $\RR$ is a probability kernel from $\cmo$ to
itself; the proof is the same.
Then, assuming also
$X_0\in\cmo$,
the \Polya{} urn process is defined by the kernel $\RR$ as before;
we have $X_n\in\cmo$ for every $n$.

\begin{example}[Drawing without replacement]\label{E-1}
Let 
$S$ be a Borel space and let $\cmo$ be the set $\cNx(S)$
of non-zero finite integer-valued
  measures on $S$; these are the measures of the type $\sum_1^m\gd_{s_i}$
  for some finite sequence $s_1,\dots,s_m$ in $S$.
(The set $\cNx(S)$ is a measurable set in $\cM(S)$, 
\eg{} as a consequence of \cite[Theorem 1.6]{Kallenberg-rm}.)
Assume that
\begin{equation}\label{e-1}
  R_s+\gd_s\in\cNx(S)
\quad\text{for every $s\in S$}.
\end{equation}
The interpretation is that the drawn ball is discarded, and instead we add
a set of balls described by the (positive, integer-valued) measure
$R_s+\gd_s$; this is thus the classical case of drawing without replacement,
see \ref{U<0} in \refS{S:intro}.
The \eqref{rso} holds, and thus a \Polya{} urn process is defined for any
initial $X_0\in\cNx(S)$.
If the urn is balanced, this is essentially the same as "$\kappa$-discrete
MVPPs" in \cite{MM}.
\end{example}

\begin{example}\label{E-2}
Let $S$ be a Borel space and consider an urn process with colour space
$S\times\oi$.
Let $\cmo:=\set{\mu\times\gl:\mu\in\cNx(S)}\subset\cMx(S\times\oi)$,
amd assume that the   replacements $R_{s,u}$ are such that
\begin{equation}
  R_{s,u}+\gd_s\times\gl \in\cmo
\quad\text{for every $s\in S$ and $u\in\oi$}.
\end{equation}
Then \eqref{rso} holds (in $S\times\oi$), and thus any $X_0\in\cmo$
defines a \Polya{} urn process.
\end{example}

\begin{remark}
  We may relax the condition $\cmo\subseteq\cMx(S)$ to
$\cmo\subseteq\cM(S)$; thus allowing $0\in\cmo$ and consequently
$R_s=-\mu$ in \eqref{rso},
which means that we remove all balls from the urn, leaving the urn empty, i.e.,
$X_{n+1}=0$. In this case, we stop the process,
  and define $X_m=0$ for all $m>n$.
Formally, $\RR$ as defined in \eqref{RR} 
then
is a probability kernel from $\cmo\setminus\set0$ to   $\cmo$; 
we extend it to a
kernel from $\cmo$ to $\cmo$ by defining $\RR_0=\gd_0$.
We leave further details for this case to the reader.
\end{remark}

\subsection{Random replacements}\label{SSdefR-}
In the random case, we similarly assume that \eqref{rso} holds \as, for
some measurable $\cmo\subseteq\cMx(S)$, every $\mu\in\cmo$ and
$\mu$-\aex{} $s$. 
We assume that $S$ is a Borel space; then \refL{Labs} implies that
$\cmo$ is a measurable subset of $\cMpm(S)$, and thus so is, for every $\mu$,
\begin{equation}
\cM_\mu:=
\bigset{\nu\in\cMpm(S):\nu+\mu\in\cmo} . 
\end{equation}
Hence the condition is that $\cR$ is a
probability kernel from $S$ to $\cMpm(S)$ such that for every
$\mu\in\cmo$,
$\cR_s(\cM_\mu)=1$
for $\mu$-\aex{} $s$. 
Then the argument in \refSS{SSdefR} shows that
$\RR$ is a probability kernel from $\cmo$ to
itself, and thus defines a \Polya{} urn process for any initial $X_0\in\cmo$.

\refT{T1} holds in this setting too, with the same proof given in
\refS{SpfT1};  the deterministic urn $\tX_n$ in $S\times\oi$ is defined as
in \refSS{SSdefD-} using
 $\tcmo:=\set{\mu\times\gl:\mu\in\cmo}\subseteq\cMx(S\times\oi)$.

\begin{example}[Random drawing without replacement]
Let $\cNx(S)$ be as in \refE{E-1} and assume that $R_s$ is a random
replacement such that \eqref{e-1} holds \as{} for every $s\in S$; as always
we assume also that $s\mapsto \cR_s:=\cL(R_s)\in\cMpm(S)$ is measurable.
Then \eqref{rso} holds \as{} for every $\mu\in\cNx(S)$ and $\mu$-\aex{} $s$,
and thus $R_s$ defines a \Polya{} urn process for any initial $X_0\in\cNx(S)$.

\refT{T1} gives an equivalent urn in $S\times\oi$ with deterministic
replacements.
Note, however, that this deterministic urn is of the type in \refE{E-2}, and
not of the simpler type in \refE{E-1}, as the random urn.
Hence, \refT{T1} may be less useful in this setting.  
\end{example}

\section*{Acknowledgement}
This work was partially done during the conference
``Modern perspectives of branching in probability''
in M\"unster, Germany, September 2017.
I thank C{\'e}cile Mailler for interesting discussions.

\newcommand\AAP{\emph{Adv. Appl. Probab.} }
\newcommand\JAP{\emph{J. Appl. Probab.} }
\newcommand\JAMS{\emph{J. \AMS} }
\newcommand\MAMS{\emph{Memoirs \AMS} }
\newcommand\PAMS{\emph{Proc. \AMS} }
\newcommand\TAMS{\emph{Trans. \AMS} }
\newcommand\AnnMS{\emph{Ann. Math. Statist.} }
\newcommand\AnnPr{\emph{Ann. Probab.} }
\newcommand\CPC{\emph{Combin. Probab. Comput.} }
\newcommand\JMAA{\emph{J. Math. Anal. Appl.} }
\newcommand\RSA{\emph{Random Struct. Alg.} }
\newcommand\ZW{\emph{Z. Wahrsch. Verw. Gebiete} }
\newcommand\DMTCS{\jour{Discr. Math. Theor. Comput. Sci.} }

\newcommand\AMS{Amer. Math. Soc.}
\newcommand\Springer{Springer-Verlag}
\newcommand\Wiley{Wiley}

\newcommand\vol{\textbf}
\newcommand\jour{\emph}
\newcommand\book{\emph}
\newcommand\inbook{\emph}
\def\no#1#2,{\unskip#2, no. #1,} 
\newcommand\toappear{\unskip, to appear}

\newcommand\arxiv[1]{\texttt{arXiv:#1}}
\newcommand\arXiv{\arxiv}

\def\nobibitem#1\par{}

\end{document}